\documentclass[11pt]{article}

\usepackage{tikz}
\usepackage{graphicx,wrapfig,lipsum}   

\usepackage{xspace}

\usepackage{amsmath,amssymb}
\usepackage{amsthm}
\usepackage{mathtools}
\usepackage[noend]{algorithmic}
\usepackage[ruled,vlined]{algorithm2e}
\usepackage{url}
\usepackage{fullpage}
\usepackage{makeidx}
\usepackage{enumerate}
\usepackage[top=1in, bottom=1.25in, left=0.75in, right=0.75in]{geometry}
\usepackage{graphicx,float,psfrag,epsfig,caption}

\usepackage{hyperref}
\usepackage{epstopdf}
\usepackage{color}
\usepackage{xr}
\usepackage{subfig}
\usepackage{caption}
\usepackage[utf8]{inputenc}

\usepackage{scalerel,stackengine}

\usepackage{enumitem}

\stackMath
\newcommand\reallywidehat[1]{%
\savestack{\tmpbox}{\stretchto{%
  \scaleto{%
    \scalerel*[\widthof{\ensuremath{#1}}]{\kern.1pt\mathchar"0362\kern.1pt}%
    {\rule{0ex}{\textheight}}
  }{\textheight}%
}{2.4ex}}%
\stackon[-6.9pt]{#1}{\tmpbox}%
}
\usepackage[mathscr]{euscript}

\DeclareSymbolFont{rsfs}{U}{rsfs}{m}{n}
\DeclareSymbolFontAlphabet{\mathscrsfs}{rsfs}

\numberwithin{equation}{section}

\newtheoremstyle{myexample} 
    {\topsep}                    
    {\topsep}                    
    {\rm }                   
    {}                           
    {\bf }                   
    {.}                          
    {.5em}                       
    {}  

\newtheoremstyle{myremark} 
    {\topsep}                    
    {\topsep}                    
    {\rm}                        
    {}                           
    {\bf}                        
    {.}                          
    {.5em}                       
    {}  

\newtheorem{claim}{Claim}[section]
\newtheorem{lemma}[claim]{Lemma}

\newtheorem{theorem}{Theorem}
\newtheorem{proposition}[claim]{Proposition}

\newtheorem{definition}[claim]{Definition}

\theoremstyle{myremark}

\theoremstyle{myremark}

\theoremstyle{myexample}

\definecolor{darkgreen}{rgb}{0.0, 0.5, 0.0}

\newcommand{\bea}{\begin{eqnarray}}
\newcommand{\eea}{\end{eqnarray}}
\newcommand{\<}{\langle}
\renewcommand{\>}{\rangle}

\def\<{\langle}
\def\>{\rangle}

\def\b0{{\boldsymbol{0}}}

\renewcommand{\b}{\mathbf{b}}

\title{Approximate Ground States of Hypercube Spin Glasses are Near Corners}
\author{Mark Sellke\\Stanford University, Department of Mathematics}
\date{}

\begin{document}

\maketitle

\abstract{\noindent We show that with probability exponentially close to $1$, all near-maximizers of any mean-field mixed $p$-spin glass Hamiltonian on the hypercube $[-1,1]^N$ are near a corner. This confirms a recent conjecture of Gamarnik and Jagannath. The proof is elementary and extends to arbitrary polytopes with $e^{o(N^2)}$ faces.}

\section{Introduction}

The present paper concerns mixed $p$-spin glasses on the hypercube $[-1,1]^N$. Such a model is specified by a sequence $\gamma_1,\gamma_2,\dots \geq 0$ of non-negative real numbers encapsulated in the mixture function \[\xi(t)=\sum_{p=1}^{\infty}\gamma_p^2 t^p.\] For each $p\in\mathbb Z^+$ we sample i.i.d. Gaussian variables $\{g_{i_1,i_2,\dots,i_p}\}_{i_1,i_2,\dots,i_p\in [N]}$ and study the resulting random Hamiltonian energy function

\[H_N(x)=\sum_{p=1}^{\infty} \frac{\gamma_p}{N^{(p+1)/2}}\sum_{i_1,\dots,i_p=1}^N g_{i_1,\dots,i_p}x_{i_1}\dots x_{i_p}. \] Equivalently, $H_N(\cdot)$ is a Gausian process with covariance \[\mathbb E[H_N(x)H_N(x')]=\frac{1}{N}\xi(\langle x,x'\rangle).\] We assume the $\gamma_p$ decay exponentially, i.e. $\lim\sup_{p\to\infty} \frac{\log \gamma_p}{p}<0$, so that there are no issues regarding convergence. Here and throughout we use a normalized inner product $\langle x,y\rangle=\frac{1}{N}\sum_i x_iy_i$ for $x,y\in\mathbb R^N$ and similarly define $|x|_2=\sqrt{\frac{1}{N}\sum_i x_i^2}.$ Hence the hypercube $[-1,1]^N$ has diameter $2$. This scaling is chosen for convenience as it makes all relevant quantities dimension-independent. We will further assume throughout that $\gamma_p$ is strictly positive for some $p\geq 2$ so that the model is a genuine spin glass. The mixture function $\xi$ is always taken to be fixed while sending $N\to\infty$.

We focus on the (random) set of near-maximizers of $H_N(x)$. This set is intimately related to the Gibbs measure $\mu(dx)\propto e^{\beta NH_N(x)}dx$ in the low temperature regime with $\beta$ large. The Gibbs measure $\mu(dx)$ is typically studied not on the continuous cube $[-1,1]^N$ but on the discrete cube $\{\pm 1\}^N$, where a great deal is known. A key quantity of interest is the free energy $F_N(\beta)=\frac{1}{N}\log\sum_{x\in \{-1,1\}^N}e^{\beta NH_N(x)}$. The limiting value (in probability) of $F_N(\beta)$ is famously given by the Parisi formula proposed in \cite{parisi1979infinite} and proved in \cite{talagrand2006parisi,panchenko2013parisi}. The existence (but not the identification) of the limiting value for large $N$ was established earlier in \cite{guerra2002thermodynamic}. 

The Hamiltonian $H_N(\cdot)$ is non-convex and may have exponentially many near-maxima \cite{chatterjee2009disorder,ding2015multiple,chen2018energy}. Moreover the structure of these near-maxima is highly nontrivial, as for each $\beta$ the Gibbs measure on $\{-1,1\}^N$ is known to concentrate on a random approximate ultrametric with high probability in so-called generic mixed $p$-spin models with $\sum_{p:\gamma_p>0}\frac{1}{p}=\infty$ \cite{jagannath2017approximate,chatterjee2019average}. Also of interest are the results \cite{auffinger2013random,auffinger2013complexity,subag2017complexity,arous2019landscape} which study the landscape of critical points for spherical spin glasses and the related spiked tensor models, computing the exponential growth rates for the number of local maxima and critical points with a given energy value. 

Let us now turn from the discrete cube $\{-1,1\}^N$ to the continuous cube $[-1,1]^N$. The free energy in this case takes a similar form as in the Ising case by the work of \cite{panchenko2005free,jagannath2020unbalanced}. Regarding the ground states, it is not difficult to see that \emph{some} near-maximum of $H_N$ on $[-1,1]^N$ must lie on a corner in $\{\pm 1\}^N$. Indeed, one may ignore the small contribution of terms of $H_N$ which are \emph{not} multi-linear and then observe that any multilinear function of the coordinates $x_1,\dots,x_N$ is maximized at some corner of the cube. However this does not rule out the existence of other near-maxima of $H_N$ which are far from a corner and therefore missed by considering the discrete cube. 

It was conjectured in \cite[Conjecture 3.6]{gamarnik2019overlap} that in fact \emph{all} near-maxima of $H_N$ on $[-1,1]^N$ must occur near the corners with high probability as $N\to\infty$. In other words, to understand the set of near-maxima of $H_N$ on $[-1,1]^N$, it is in some sense sufficient to understand it on the discrete cube. Conditional on (an implication of) this result, \cite{gamarnik2019overlap} prove that approximate message passing algorithms fail to approximately optimize pure $p$-spin models with $\gamma_p\neq 0$ for exactly $1$ value of $p$, over $[-1,1]^N$ when $p\geq 4$ is even. Moreover their proof seems to apply to any $\xi$ satisfying a suitable \emph{overlap gap} property, perhaps with the requirement $\gamma_1=0$. By contrast for mixture functions $\xi$ satisfying a strong \emph{no overlap gap} condition, approximate message passing algorithms are able to efficiently locate near-maxima of $H_N$ with high probability \cite{mon18,ams20}.

Our main result is that all near-maxima of $H_N$ on $[-1,1]^N$ are close to a corner in $\{\pm 1\}^N$, confirming the conjecture of \cite{gamarnik2019overlap}. Moreover we obtain an explicit quantitative dependence, though we do not expect it to be tight. Below we use the notation $\Omega_{\varepsilon,\eta}(N)$ to represent a quantity bounded below by $C(\varepsilon,\eta)N$ for some constant $C(\varepsilon,\eta)$ independent of $N$ when $N\geq N_0(\varepsilon,\eta)$ is sufficiently large.

\begin{theorem}

\label{thm:main}

Let $\xi$ define a mixed $p$-spin model and fix $\varepsilon,\eta>0$. Then for $N$ sufficiently large, with probability $1-e^{-\Omega_{\varepsilon,\eta}(N)}$ all $x\in [-1,1]^N$ with \[H_N(x)\geq \max_{y\in [-1,1]^N}H_N(y)-\int_{1-\varepsilon}^1 \sqrt{(1-t)\xi''(t)}dt+\eta\] satisfy $|x|_2^2\geq 1-\varepsilon.$ 

\end{theorem}

The idea of the proof is based on that of \cite{subag2018following}, which uses uniform control of top eigenvalues of the Hessian $\nabla^2 H_N(x)$ to optimize mean field spin glasses on the sphere via small local steps. Our main insight is that this idea continues to work when a constant fraction of coordinates are fixed at $\pm 1$, allowing us to substantially increase the energy $H_N$ from any starting point far from a corner even after reaching the boundary of $[-1,1]^N$. Our proof is elementary and avoids any reliance on complicated Parisi-type variational formulas which characterize much of the spin glass literature. In fact it does not even require the existence of a limiting value for $\max_{y\in [-1,1]^N}H_N(y)$. Due to the simplicity of our approach, Theorem~\ref{thm:main} extends to quite general polytopes as we explain in Section~\ref{sec:gen}.

Finally let us mention two alternative approaches to our main result, at least on the cube. First, we believe that \cite[Theorems 8, 9]{chen2019generalized} should imply Theorem~\ref{thm:main}. In their language it suffices to check that $TAP^{\infty}(\mu)$ is bounded away from $0$ for $\mu$ a probability measure on $[-1,1]$ with $L^2$ norm bounded away from $1$. See also \cite{chen2018generalized,subag2018free} for positive temperature and spherical analogs. Second, \cite[Theorem 1.2]{jagannath2020unbalanced} gives a Parisi-type formula for the ground state energy of mixed $p$-spin models on the subset of $[-1,1]^N$ with any asymptotically fixed $L^2$ norm in $[0,1]$, so showing that their formula is strictly increasing in this $L^2$ norm would imply Theorem~\ref{thm:main}. However a proof produced by either method would be far less elementary than the proof we present, and even with significant effort might not extend beyond highly structured classes of polytopes.

\section{Proof of Theorem~\ref{thm:main}}

We set $\zeta(t)=\sqrt{\xi''(t)}.$ By our assumption that $\gamma_p>0$ for some $p\geq 2$ it follows that $\zeta(t)>0$ for any $t>0$. Below and throughout, given $M\in Mat_{N\times N}(\mathbb R)$ and a subspace $W\subseteq \mathbb R^N$ we set $M|_W=P_W^{\top}MP_W$ where $P_W:\mathbb R^N\to W$ is the orthogonal projection onto $W$. In other words $M|_W$ is the restriction of $M$ to $W$ as a bilinear form, and is a matrix of size $\dim(W)\times \dim(W)$. 

\begin{proposition}
\label{prop:GOE}
For nonzero $x\in\mathbb R^N$ let $x^{\perp}$ denote the orthogonal subspace to $x$. For any fixed subspace $W\subseteq x^{\perp}$, the restriction $\nabla^2 H_N(x)|_{W}$ of the Hessian of $H_N$ to $x^{\perp}$ has the distribution of a $GOE(\dim(W))$ matrix times $\zeta(|x|_2^2)\sqrt{\frac{\dim(W)}{N}}.$

\end{proposition}

By a $GOE(N)$ matrix we mean a symmetric $N\times N$ matrix of independent centered Gaussians in which diagonal entries have variance $\frac{2}{N}$ and off-diagonals have variance $\frac{1}{N}$. In the case $W=x^{\perp}$, Proposition~\ref{prop:GOE} can be shown as in \cite[Equation (3.10)]{subag2018following} by setting $x=(x_1,0,\dots,0)$ using rotational invariance and performing a simple direct computation. See also the text following \cite[Equation (1.8)]{subag2018following}. Proposition~\ref{prop:GOE} then follows for general subspaces $W\subseteq x^{\perp}$ because the upper $\dim(W)\times \dim(W)$ corner of a $GOE(N-1)$ matrix is a $GOE(\dim(W))$ matrix up to scaling.

Denote the eigenvalues of a symmetric matrix $G$ in decreasing order by $\lambda_1(G)\geq\lambda_2(G)\geq\dots$. Recall that $\lambda_1(G)\approx 2$ holds with high probability when $G\sim GOE(N)$. In fact the following fundamental result states that many eigenvalues are at least $2-\delta$ with extremely high $1-e^{-\Omega_{\delta}(N^2)}$ probability. It follows from \cite[Theorem 1.1]{arous1997large} and is also used in the proof of \cite[Lemma 3]{subag2018following}. See also \cite[Theorem 2.6.1]{anderson2010introduction}. 

\begin{proposition}

\label{prop:LDP}

For any $\delta>0$ and fixed positive integer $k$, if $G\sim GOE(N)$ then \[\mathbb P[\lambda_k(G)\geq 2-\delta]\geq 1-e^{-\Omega_{\delta,k}(N^2)}.\]

\end{proposition}

We also require the following apriori uniform bound on the derivatives of $H_N$ taken from \cite{arous2020geometry}. Below $\mathbb{B}^{N}$ denotes the unit ball $\{\boldsymbol{\sigma}\in\mathbb R^N:|\boldsymbol{\sigma}|_2\leq 1\}$ while $\mathbb S^{N-1}$ denotes the unit sphere $\{v\in\mathbb R^N:|v|_2= 1\}$. Here again we use the rescaled norm in which $|v|_2=\sqrt{\frac{\sum_i v_i^2}{N}}$; moreover our definition of $H_N$ differs from that of \cite{arous2020geometry} by a factor $N$. This is why the derivative estimates below are of constant order unlike in \cite{arous2020geometry}.

\begin{lemma}{\cite[Corollary 59]{arous2020geometry}}

\label{lem:lip} Let $H_N$ be the Hamiltonian for a mixed $p$-spin model with fixed mixture $\xi$ satisfying $\lim_{p\to\infty} \frac{\log \gamma_p}{p}<0$. For appropriate $C>0$ and $i=1,2,3$ we have:

\[\mathbb{P}\left[\forall \boldsymbol{\sigma} \in \mathbb{B}^{N}, \forall v\in\mathbb S^{N-1}:\left|\partial_v^i H_{N}(\boldsymbol{\sigma})\right|<C\right] \geq 1-e^{-\Omega(N)}.\]

\[\mathbb{P}\left[\forall \boldsymbol{\sigma}, \boldsymbol{\sigma}^{\prime} \in \mathbb{B}^{N}:\left\|\nabla^{2} H_{N}(\boldsymbol{\sigma})-\nabla^{2} H_{N}\left(\boldsymbol{\sigma}^{\prime}\right)\right\|_{o p}<C\left\|\boldsymbol{\sigma}-\boldsymbol{\sigma}^{\prime}\right\|\right] \geq 1-e^{-\Omega(N)}.\] 

\end{lemma}

We next define the class of axis-aligned subspaces $W_S$ for $S\subseteq [N]$. The key to our proof is to obtain uniform control on the Hessians $H_N(x)|_{W_S}$ over all $x\in [-1,1]^N$ and large $S$.

\begin{definition}

Given a subset $S\subseteq [N]$ we denote by $W_S$ the $|S|$ dimensional subspace spanned by elementary basis vectors $e_s$ for $s\in S$. We set $W_S(x)=W_S\cap x^{\perp}$ so that $\dim(W_S(x))\in \{|S|-1,|S|\}$.

\end{definition}

\begin{definition}

The Hamiltonian $H_N:[-1,1]^N\to\mathbb R$ is $(\varepsilon,\delta)$-good at $x\in [-1,1]^N$ if for every subset $S\subseteq [N]$ of size $|S|\geq \varepsilon N$,

\[\lambda_1\big(\nabla^2H_N(x)|_{W_S(x)}\big)\geq  2\zeta(|x|_2^2)\sqrt{\varepsilon}-\delta.\] $H_N$ is $(\varepsilon,\delta)$-good if it is $(\varepsilon,\delta)$-good at all $x\in [-1,1]^N$, and is $\delta$-good if it is $(\varepsilon,\delta)$-good for all $\varepsilon\geq \delta$.

\end{definition}

Roughly speaking, $H_N$ is $\delta$-good if its Hessian has a maximum eigenvalue of typical size or larger on all high-dimensional axis-aligned affine subspaces. We next show this condition occurs with exponentially high probability.

\begin{lemma}

\label{lem:hamiltoniangood}

Fix $\delta>0$ and $\xi(t)=\sum_{p\geq 1}\gamma_p^2 t^p$. Then $H_N(\cdot)$ is $\delta$-good with probability $1-e^{-\Omega_{\delta}(N)}$.

\end{lemma}

\begin{proof}


We follow the proof of the case $S=[N]$ in \cite[Lemma 3]{subag2018following}, extending the union bound to be over subsets $S$ as well as points $x$. First, it suffices to show $H_N$ is $(\varepsilon,\delta/2)$-good with the claimed probability for all fixed $(\varepsilon,\delta)$ since one can then union bound over $O_{\delta}(1)$ values of $\varepsilon$ using uniform continuity of $\zeta$. Replacing $\delta/2$ by $\delta$, we will show that $H_N$ is $(\varepsilon,\delta)$-good with probability $1-e^{-\Omega_{\varepsilon,\delta}(N)}$ which implies the conclusion.

For any fixed $x\in [-1,1]^N$ and $S\subseteq [N]$, because $W_S(x)\subseteq x^{\perp}$, we obtain from Proposition~\ref{prop:GOE} that the restricted Hessian $\nabla^2 H_N(x)|_{W_S(x)}$ has the law of $\zeta(|x|_2^2)\sqrt{\frac{\dim W_S(x)}{N}}\cdot GOE(\dim W_S(x))$. As $\dim W_S(x)\geq |S|-1$, Proposition~\ref{prop:LDP} implies:

\begin{equation}
\label{eq:eigenlowerbound}
    \mathbb P\left[\lambda_2(\nabla^2H_N(x)|_{W_S(x)})\geq 2\zeta(|x|_2^2)\sqrt{\frac{|S|-1}{N}}-\frac{\delta}{2}\right]\geq 1-e^{-\Omega_{\delta}(|S|^2)}.
\end{equation} Restricting to $|S|\geq \varepsilon N$ and observing there are at most $2^N$ possibilities for $S$, we conclude that for any fixed $x\in [-1,1]^N$,

\begin{equation}\label{eq:eigenlowerboundv2}\mathbb P\left[\forall S\subseteq [N],|S|\geq \varepsilon N:  \lambda_2(\nabla^2H_N(x)|_{W_S(x)})\geq 2\zeta(|x|_2^2)\sqrt{\frac{|S|-1}{N}}-\frac{\delta}{2} \right]\geq 1-e^{-\Omega_{\varepsilon,\delta}(N^2)}.\end{equation}




Next choose a $\delta'$-net $\mathcal N_{\delta'}$ for $[-1,1]^N$ of size $|\mathcal N_{\delta'}|=e^{O_{\delta'}(N)}$. Union bounding over $y\in \mathcal N_{\delta'}$, it follows that \eqref{eq:eigenlowerboundv2} holds for all $y\in\mathcal N_{\delta'}$ simultaneously with the same high probability $1-e^{-\Omega_{\varepsilon,\delta}(N^2)}$. Assume additionally that the conclusions of Lemma~\ref{lem:lip} hold, which is with probability $1-e^{-\Omega(N)}$. Under these conditions we now show that Equation~\eqref{eq:eigenlowerboundv2} holds simultaneously for all $x\in [-1,1]^N$ and $|S|\geq \varepsilon N$. For such an $x$ choose $y=y(x)\in\mathcal N_{\delta'}$ with $|x-y|_{2}\leq \delta'$. Difference of squares and the triangle inequality imply

\[|x|_2^2-|y|_2^2 \leq 2|x-y|_2 \leq 2\delta'.\]

From Lemma~\ref{lem:lip} and the fact (which follows from the Courant-Fisher characterization) that $|\lambda_k(M)-\lambda_k(M')|\leq |M-M'|_{op}$ for any symmetric matrices $M,M'$ and any integer $k$, 

\begin{align*}\left|\lambda_2\left(\nabla^2H_N(x)|_{W_S(y)}\right)-\lambda_2\left(\nabla^2H_N(y)|_{W_S(y)}\right)\right|&\leq \left|\left(\nabla^2H_N(x)-\nabla^2H_N(y)\right)|_{W_S(y)}\right|_{op} \\
&\leq \left|\nabla^2H_N(x)-\nabla^2H_N(y)\right|_{op} \\
&\leq 2C\delta'.
\end{align*}

From here we derive the eigenvalue lower bound

\begin{align*}
    \lambda_1\left(\nabla^2 H_N(x)|_{W_S(x)}\right)&\geq \lambda_1\left(\nabla^2 H_N(x)|_{W_S(x)\cap y^{\perp}}\right)\\
    &\geq \lambda_2\left(\nabla^2 H_N(x)|_{W_S(y)}\right)\\
    &\geq \lambda_2\left(\nabla^2 H_N(y)|_{W_S(y)}\right)-2C\delta'\\
    &\geq 2\zeta(|y|_2^2)\sqrt{\varepsilon}-\frac{\delta}{2}-2C\delta'\\
    &\geq 2\zeta(\max(0,|x|_2^2-2\delta'))\sqrt{\varepsilon}-\frac{\delta}{2}-2C\delta'.
\end{align*}

As $\zeta$ is uniformly continuous on $[0,1]$, taking $\delta'$ sufficiently small gives the conclusion

\[\lambda_1\left(\nabla^2 H_N(x)|_{W_S(x)}\right)\geq 2\zeta(|x|_2^2)\sqrt{\varepsilon}-\delta.\]

Because $x\in [-1,1]^N$ and $|S|\geq\varepsilon N$ were arbitrary, we conclude that $H_N$ is $(\varepsilon,\delta)$-good with probability $1-e^{-\Omega(N)}-e^{-\Omega_{\varepsilon,\delta}(N^2)}=1-e^{-\Omega_{\varepsilon,\delta}(N)}$. Recalling the discussion at the beginning of the proof, it follows that $H_N$ is $\delta$-good with probability $1-e^{-\Omega_{\delta}(N)}$ as claimed.

\end{proof}

The next lemma shows how to use Lemma~\ref{lem:hamiltoniangood} to obtain local improvements to $H_N(\cdot)$ from any point $x\in [-1,1]^N$ which is far from a corner. 

\begin{lemma}

\label{lem:increment}

Suppose the Hamiltonian $H_N$ is $\delta$-good and satisfies the guarantee of Lemma~\ref{lem:lip}. Then for any $x\in [-1,1]^N$ with $|x|_2^2\leq 1-\delta$ there is a non-zero vector $v$ orthogonal to $x$ such that:
\begin{enumerate}
    \item $x+v\in [-1,1]^N$
    \item If $|x_i|=1$ then $v_i=0$.
    \item \[H_N(x+v)-H_N(x)\geq \left(\zeta(|x|_2^2)\sqrt{1-|x|_2^2}-\delta\right)|v|_2^2.\]
    \item $|v|_2\leq \frac{\delta}{10C}.$
    \item Either $|v|_2=\frac{\delta}{10C}$ or $x+v$ has strictly more $\pm 1$-valued coordinates than $x$.
\end{enumerate} 
\end{lemma}

\begin{proof}

By a simple Markov inequality we know that $x$ has a set $S$ of at least $(1-|x|_2^2) N$ coordinates not equal to $\pm 1$. Because $H_N$ is $\delta$-good the restriction $\nabla^2H_N(x)|_{W_S(x)}$ has an eigenvalue at least $2\zeta(|x|_2^2)\sqrt{1-|x|_2^2}-\delta $ with corresponding eigenvector $v$. Since $v\in W_S\subseteq\mathbb R^N$, we may by slight abuse of notation treat $v$ as a vector in $\mathbb R^N$. Of course this $v\in\mathbb R^N$ need not be an eigenvector of $\nabla^2 H_N(x)$ but we retain the Rayleigh quotient lower bound 

\[\langle v,\nabla^2 H_N(x)v\rangle\geq \left(2\zeta(|x|_2^2)\sqrt{1-|x|_2^2}-\delta\right) |v|_2^2.\] Since $v,-v$ play symmetric roles we may assume by symmetry that $\langle\nabla H_N(x),v\rangle\geq 0$. By scaling $v$ to be sufficiently small we may assume that $x+v\in [-1,1]^N$ and that $|v|_2\leq \frac{\delta}{10C}.$ Using the guarantee of Lemma~\ref{lem:lip} with $i=3$, it follows that along the line segment $x+[0,1] v$ the Hessian of $H_N$ varies in operator norm by at most $\frac{\delta}{5}$. This combined with $\langle\nabla H_N(x),v\rangle\geq 0$ easily implies that 

\[H_N(x+v)\geq H_N(x)+\left(\zeta(|x|_2^2)\sqrt{1-|x|_2^2}-\delta\right)|v|_2^2. \]

Hence $v$ satisfies the first 4 claimed conditions. By scaling $v$ to be as long as possible given the constraints $x+v\in [-1,1]^N$ and $|v|_2\leq \frac{\delta}{10C}$ we ensure that item 5 is satisfied.

\end{proof}

\begin{proof}[Proof of Theorem~\ref{thm:main}]

We take $\delta$ small depending on $\varepsilon$ and assume $H_N$ is $\delta$-good and that the conclusion of Lemma~\ref{lem:lip} holds. For any point $x^0\in [-1,1]^N$ with $|x^0|_2^2\leq 1-\varepsilon$ we choose $v^0$ as guaranteed by Lemma~\ref{lem:increment} and set $x^1=x^0+v^0$. We continue producing iterates $x^{i+1}=x^i+v^i$ via Lemma~\ref{lem:increment} with increasing energies until we reach an $x^m$ with $|x^m|_2^2\geq 1-\delta$. By part 5 of Lemma~\ref{lem:increment}, this occurs for some finite $m$.

Since $v^i$ is orthogonal to $x^i$, we find

\begin{align}H_N(x^m)-H_N(x^0)&= \sum_{i<m} H_N(x^{i+1})-H_N(x^i)\\
&\geq \sum_{i<m} (\zeta(|x^i|_2^2)\sqrt{1-|x^i|_2^2}-\delta)|v^i|_2^2\\
&= \sum_{i<m} (\zeta(|x^i|_2^2)\sqrt{1-|x^i|_2^2}-\delta)(|x^{i+1}|_2^2-|x^i|_2^2). \end{align}

Up to the error $\sum_{i<m}\delta (|x^{i+1}|_2^2-|x^i|_2^2)\leq \delta$, this is exactly a Riemann sum for the integral $\int_{|x^0|_2^2}^{|x^m|_2^2} \zeta(t)\sqrt{1-t}dt$. Because $|v^i|_2^2\to 0$ as $\delta\to 0$ uniformly in $i$, and $|x^0|_2^2\leq 1-\varepsilon, |x^m|_2^2\geq 1-\delta$, these Riemann sums have limit infimum at least the integral $\int_{1-\varepsilon}^{1} \zeta(t)\sqrt{1-t}dt$. Hence for fixed $\varepsilon$, and $\delta\to 0$, we obtain \[H_N(x^m)-H_N(x^0)\geq \int_{1-\varepsilon}^1 \zeta(t)\sqrt{1-t}dt -o_{\delta\to 0}(1).\] Here $o_{\delta\to 0}(1)$ indicates a term tending to $0$ as $\delta\to 0$, uniformly in $N$. Since $x^0$ was arbitrary given the constraint $|x^0|_2^2\leq 1-\varepsilon$ and $H_N(x^m)\leq \max_{y\in [-1,1]^N}H_N(y)$, taking $\delta$ small enough depending on $(\varepsilon,\eta)$ concludes the proof.

\end{proof}

\section{Extension to General Polytopes}

\label{sec:gen}

Theorem~\ref{thm:main} extends to more general polytopes than cubes. In particular we show that for bounded polytopes with $e^{o(N^2)}$ total faces, all near-maxima of $H_N$ over the polytope occur near a point at which $(1-\varepsilon)N$ faces are incident. We remark that the condition of $e^{o(N^2)}$ total faces is implied by having either $e^{o(N)}$ vertices or $e^{o(N)}$ maximal (i.e. codimension $1$) faces. It also holds for any product of $O(N)$ polytopes of constant dimension.

\begin{definition}

A sequence of polytopes $\mathcal P_N\subseteq \mathbb R^N$ is said to be \emph{regular} if:

\begin{enumerate}
    \item $\mathcal P_N$ has at most $e^{o(N^2)}$ faces of all dimensions.
    \item $\mathcal P_N\subseteq\mathbb B_N$.
    
\end{enumerate}

\end{definition}

The second condition ensures that Lemma~\ref{lem:lip} continues to hold over $\mathcal P_N$. We again remind the reader that we use the normalization $|x|_2^2=\frac{\sum_{i=1}^N x_i^2}{N}$

\begin{definition}

Given a polytope $\mathcal P_N\subseteq \mathbb R^N$ an $\varepsilon$-corner is a point on the boundary $\partial \mathcal P_N$ at which at least $(1-\varepsilon)N$ faces meet.

\end{definition}

\begin{theorem}
\label{thm:main2}
Let $\xi$ define a mixed $p$-spin model and fix $\varepsilon,\varepsilon',\eta>0$. Let $\mathcal P_N\subseteq \mathbb R^N$ be a regular sequence of polytopes. Then for $N$ sufficiently large, with probability $1-e^{-\Omega_{\varepsilon,\varepsilon',\eta}(N)}$, for any $x\in \mathcal P_N$ satisfying

\[H_N(x)\geq \max_{y\in \mathcal P_N}H_N(y)-\sqrt{\varepsilon}\int_{|x|_2^2}^{|x|_2^2+\varepsilon'} \zeta(t) dt+\eta\] there exists an $\varepsilon$-corner $\hat x$ of $\mathcal P_N$ with $|x-\hat x|_2^2\leq \varepsilon'.$ 

\end{theorem}

Note that because $\zeta$ is increasing, we have $\int_{|x|_2^2}^{|x|_2^2+\varepsilon'}\zeta(t)dt\geq \int_0^{\varepsilon'}\zeta(t)dt$ which is positive and independent of $x$. The proof is almost the same as the cubical case. The subspaces $W_S$ are replaced by the family of all $\mathcal P_N$-face-aligned subspaces in $\mathbb R^N$. The main difference is that to prove Theorem~\ref{thm:main2} it does not suffice to track the distance $|x^i|_2$ to the origin, as being a near-corner is no longer equivalent to having near-maximal distance from the origin. Because of this we additionally track the distances $|x^i-x^0|_2$ of our sequence $x^0,x^1,x^2,\dots$ from the starting point $x^0$. This leads to an additional linear constraint on the increment vectors $v^i$ and hence requires one more large eigenvalue of the restricted Hessians. 

\begin{definition}

We say a subspace $U\subseteq\mathbb R^N$ is $\mathcal P_N$-\emph{face-aligned} if $\mathcal P_N$ has a face whose tangent space is exactly (a translate of) $U$. 

\end{definition}

\begin{definition}

For a mixture $\xi$ and polytope $\mathcal P_N\subseteq\mathbb R^N$, the Hamiltonian $H_N$ is $(\varepsilon,\delta)$-superb if for all $x\in\mathcal P_N$ and all $\mathcal P_N$-face-aligned subspaces $U$ with $\dim(U)\geq \varepsilon N$,

\[\lambda_2(\nabla^2H_N(x)|_{U\cap x^{\perp}})\geq  2\zeta(|x|_2^2)\sqrt{\varepsilon}-\delta.\]

\end{definition}

\begin{lemma}
\label{lem:hamstillgood}

Fix $\varepsilon,\delta>0$ a mixture $\xi$, and a regular sequence $\mathcal P_N$ of polytopes. Then with probability $1-e^{-\Omega_{\varepsilon,\delta}(N)}$ the random function $H_N$ is $(\varepsilon,\delta)$-superb.

\end{lemma}

\begin{proof}

The proof is almost identical to that of Lemma~\ref{lem:hamiltoniangood} - note that $U\cap x^{\perp}$ is exactly the same as $W_S(x)$ for the case of the cube. To obtain a lower bound on $\lambda_2$ rather than $\lambda_1$, we simply change all instances of $\lambda_k$ to $\lambda_{k+1}$ in the proof of Lemma~\ref{lem:hamiltoniangood}. Regularity of $\mathcal P_N$ ensures that when we take a union bound over pairs $(y,U)$ for $y$ in a $\delta'$-net $\mathcal N_{\delta'}\subseteq\mathcal P_N$ and all $\mathcal P_N$-face-aligned subspaces $U$, we only consider $e^{o(N^2)}$ distinct pairs. Hence the $N^2$ large deviation rate of Proposition~\ref{prop:LDP} ensures uniform eigenvalue lower bounds across all such pairs with exponentially high probability. As remarked previously, Lemma~\ref{lem:lip} continues to apply to $\mathcal P_N$, so that by again taking $\delta'$ small we extend from a $\delta'$ net to all of $\mathcal P_N$ just as in Lemma~\ref{lem:hamiltoniangood}. 

\end{proof}

We next give the analog of Lemma~\ref{lem:increment} with the new constraint $\langle v,x^0\rangle=0$.

\begin{lemma}

\label{lem:incrementv2}

Suppose the Hamiltonian $H_N$ is $(\varepsilon,\delta)$-superb and satisfies the guarantee of Lemma~\ref{lem:lip}. Then for any $x\in \mathcal P_N$ not a $\varepsilon$-corner and for any $x^0\in\mathcal P_N$ there is a non-zero vector $v$ orthogonal to both $x$ and $x^0$ such that:
\begin{enumerate}
    \item $x+v\in \mathcal P_N$
    \item If $x$ is contained in a boundary face of $\mathcal P_N$, then $x+v$ is in the same face.
    \item \[H_N(x+v)-H_N(x)\geq \left(\zeta(|x|_2^2)\sqrt{\varepsilon}-\delta\right)|v|_2^2.\]
    \item $|v|_2\leq \frac{\delta}{10C}.$
    \item Either $|v|_2=\frac{\delta}{10C}$ or $x+v$ is contained in a face of dimension strictly smaller than that of any face containing $x$.
\end{enumerate} 
\end{lemma}

\begin{proof}

Let $U$ be the subspace corresponding to the minimal face containing $x$. As $x$ is not an $\varepsilon$-corner we know that $|U|\geq\varepsilon N.$ By Cauchy interlacing, 

\begin{align}\lambda_1\left(\nabla^2 H_N(x)|_{U\cap x^{\perp}\cap x^{0\perp}}\right)&\geq \lambda_2\left(\nabla^2 H_N(x)|_{U\cap x^{\perp}}\right)\\
&\geq 2\zeta(|x|_2^2)\sqrt{\varepsilon}-\delta.\end{align}

Let $v\in U\cap x^{\perp}\cap x^{0 \perp}$ be the top eigenvector of $\nabla^2 H_N(x)$. As before since $U\cap x^{\perp}\cap x^{0 \perp}\subseteq\mathbb R^N$ we may treat $v$ as a vector in $\mathbb R^N$. The remainder of the proof is identical to that of Lemma~\ref{lem:increment}.

\end{proof}

\begin{proof}[Proof of Theorem~\ref{thm:main2}]

We begin with a point $x^0$ and repeatedly us Lemma~\ref{lem:incrementv2}, whose assumptions hold with probability $1-e^{-\Omega_{\varepsilon,\delta}(N)}$. We obtain a sequence $x^0,x^1=x^0+v^i,x^2=x^1+v^1,\dots$ of points in our polytope. We continue until reaching an $\varepsilon$-corner $x^m$. We have for each $i$:

\[H_N(x^{i+1})-H_N(x^i)\geq (\zeta(|x^i|_2^2)\sqrt{\varepsilon}-\delta)|v^i|_2^2. \]

From the orthogonality conditions on $v^i$ we have

\begin{equation}\label{eq:ortho}|x^{i+1}|_2^2-|x^i|_2^2=|v^i|_2^2=|x^{i+1}-x^0|_2^2-|x^{i}-x^0|_2^2\end{equation}
which implies $\delta\sum_{i=0}^{m-1} |v^i|_2^2=O(\delta)=o_{\delta\to 0}(1).$  As in the proof of Theorem~\ref{thm:main}, the fact that $|v^i|_2^2\to 0$ uniformly as $\delta\to 0$ gives the Riemann sum convergence

\[\sum_{i=0}^{m-1} \zeta(|x^i|_2^2)\sqrt{\varepsilon}|v^i|_2^2=\sum_{i=0}^{m-1} \zeta(|x^i|_2^2)\sqrt{\varepsilon}\bigg(|x^{i+1}|_2^2-|x^{i}|_2^2\bigg)\to \sqrt{\varepsilon}\int_{|x^0|_2^2}^{|x^m|_2^2} \zeta(t) dt.\]

It follows that for $\delta$ sufficiently small as a function of $\varepsilon$, with probability $1-e^{-\Omega_{\varepsilon,\delta}(N)}$:

\begin{align*}H_N(x^m)-H_N(x^0)&\geq \sqrt{\varepsilon}\int_{|x^0|^2}^{|x^m|_2^2} \zeta(t) dt-o_{\delta\to 0}(1)\\
&=\sqrt{\varepsilon}\int_{|x^0|_2^2}^{|x^0|_2^2+|x^m-x^0|_2^2} \zeta(t)dt-o_{\delta\to 0}(1). \end{align*}

Here the latter equality follows from~\eqref{eq:ortho}. Since $x^0$ was arbitrary and $x^m$ is an $\varepsilon$-corner, taking $\delta$ small enough depending on $(\varepsilon,\varepsilon',\eta)$ completes the proof of Theorem~\ref{thm:main2}. 

\end{proof}

\section*{Acknowledgement}

This work was supported by NSF and Stanford graduate research fellowships and was conducted while the author was visiting the Simons Institute for the Theory of Computing. We thank David Gamarnik for bringing this problem to our attention.

\bibliographystyle{alpha}
\bibliography{all-bib}

\end{document}